\documentclass[a4paper,12pt]{amsart}

\usepackage{amssymb}
\usepackage{latexsym}
\usepackage{amsmath}
\usepackage{rotating}
\usepackage[all]{xy}
\usepackage[latin1]{inputenc}
\usepackage[usenames,dvipsnames]{color}
\usepackage{amsthm,amsfonts,amsxtra,amssymb,amscd}
\usepackage{fullpage}

\newcommand{\spmo}{{\mathbb S}}
\newcommand{\mC}{{\mathbb C}}

\newcommand{\mP}{{\mathbb P}}

\newcommand{\gra}{\alpha}         
\newcommand{\grd}{\delta}  \newcommand{\gre}{\varepsilon} 
  \newcommand{\grl}{\lambda}     \newcommand{\grs}{\sigma}
\newcommand{\grf}{\varphi}

 \newcommand{\goX}{\mathfrak X}

\newcommand{\gob}{\mathfrak b}  
  \newcommand{\gog}{\mathfrak g}
\newcommand{\goh}{\mathfrak h}  
  
 \newcommand{\goo}{\mathfrak o} 
  \newcommand{\gos}{\mathfrak s}
\newcommand{\got}{\mathfrak t}

  \newcommand{\calG}{\mathcal G}
 \newcommand{\calI}{\mathcal I} 
\newcommand{\calK}{\mathcal K} \newcommand{\calL}{\mathcal L} 
  
 \newcommand{\calR}{\mathcal R} \newcommand{\calS}{\mathcal S}
 \newcommand{\calU}{\mathcal U}

\newcommand{\sfA}{\mathsf A}
\newcommand{\sfB}{\mathsf B} \newcommand{\sfC}{\mathsf C} 
\newcommand{\sfE}{\mathsf E}

  \newcommand{\sfS}{\mathsf S}

  \newcommand{\mtd}{\mathrm d}

\def\N{{\mathbb N}}
\def\Z{{\mathbb Z}}

\def\C{{\mathbb C}}

\def\Id{\mathop{\rm Id}\nolimits}

\def\Spin{\mathop{\rm Spin}\nolimits}

\newcommand{\gbin}[3]{{#1 \brack #2}_{#3}}
\newcommand{\pf}{{\mathop{\rm pf}\nolimits}}

\def\lieg{\mathfrak{g}}

\newcommand{\Pf}   {\operatorname{Pf}}
\newcommand{\wt}   {\operatorname{wt}}
\newcommand{\End}  {\operatorname{End}}
\newcommand{\Hom}  {\operatorname{Hom}}
\newcommand{\Tr}   {\operatorname{Tr}}

\newcommand{\LGr}  {{\calG r}}

\newcommand{\mfor}     {\text{ for }} 
\newcommand{\mif}      {\text{ if }}

\newcommand{\lra}      {\longrightarrow}
\newcommand{\isocan}   {\simeq}
\newcommand{\vuoto}    {\varnothing}

\renewcommand{\geq}    {\geqslant}
\renewcommand{\leq}    {\leqslant}
\newcommand{\senza}    {\smallsetminus}

\def\emptyRow{{\varnothing}}

\DeclareMathAlphabet{\mathpzc}{OT1}{pzc}{m}{it}
\newcommand{\calp}{\mathpzc{p}}

\newcommand{\verleq}{\begin{sideways}$\geq$\end{sideways}}
\newcommand{\vergreater}{\begin{sideways}$<$\end{sideways}}
\newcommand{\rowsSymbol}{\mathcal{R}}
\newcommand{\standardRowsSymbol}{\mathcal{SR}}
\newcommand{\rows}[1]{\rowsSymbol_{#1}}
\newcommand{\distintRows}[1]{\rowsSymbol^+_{#1}}
\newcommand{\standardRows}[1]{\standardRowsSymbol_{#1}}
\newcommand{\evenRows}[1]{\rowsSymbol^0_{#1}}
\newcommand{\evenDistintRows}[1]{\rowsSymbol^{+,0}_{#1}}
\newcommand{\evenStandardRows}[1]{\standardRowsSymbol^0_{#1}}
\newcommand{\replace}{\longrightarrow}
\newcommand{\myLabel}[1]{\label{#1}}

\theoremstyle{definition}
\newtheorem{definition}{Definition}
\newtheorem{remark}[definition]{Remark}
\newtheorem{example}[definition]{Example}
\theoremstyle{theorem}
\newtheorem{lemma}[definition]{Lemma}
\newtheorem{proposition}[definition]{Proposition}
\newtheorem{theorem}[definition]{Theorem}
\newtheorem*{theorem_}{Theorem}
\newtheorem{corollary}[definition]{Corollary}

\title{Pfaffians and shuffling relations for the spin module}

\author{R. Chiriv\`\i\ and A. Maffei}

\begin{document}

\maketitle

\begin{abstract}
We present explicit formulas for a set of generators of the ideal of
relations among the pfaffians of the principal minors of the
antisymmetric matrices of fixed dimension. These formulas have an
interpretation in terms of the standard monomial theory for the spin
module of orthogonal groups.
\end{abstract}

\section*{Introduction}

Let $\goX_{n+1}$ be the set of $(n+1)\times(n+1)$ antisymmetric
matrices over the complex number. It is well known that the determinant $\goX_{n+1}\ni X\longmapsto\det X\in\C$
is the square of a polynomial function $\goX_{n+1}\ni
X\longmapsto\Pf(X)\in\C$ called the \emph{pfaffian} of a
matrix. In particular only the even dimensional principal minors of
$X\in\goX_{n+1}$ have non zero pfaffian.

Let $\sfB$ be the ring $\C[x_{ij}\,|\,1\leq i<j\leq n+1]$. As a set of generators of $\sfB$ we choose the pfaffians of the (even dimensional) principal
minors of matrices in $\goX_{n+1}$.  In this paper we want to describe
the relations among these generators. This is a classical problem
and may be seen as analogous to the Pl\"ucker relations for
determinants. Indeed the formulas we present are very similar to
Pl\"ucker formulas. In order to state such formulas we introduce some notations.

The pfaffians of the principal minors are indexed by lists of even
length of integers in $\{1,2,\ldots,n+1\}$, or, equivalently, by
\emph{row}, i.e. list $I=i_1i_2\cdots i_r$ (of any length) of integers
in $\{1,2,\ldots,n\}$ in the following way: define $I^0$ be the either
the same sequence $I$ if $r$ is even or the sequence $i_1i_2\cdots
i_r,n+1$ if $r$ is odd and let $\pf_I$ be the polynomial function on
$\goX_{n+1}$ given by the pfaffian of the principal minor with rows
and columns in $I^0$. For convenience we set also $\pf_\vuoto$ to be
equal to the constant polynomial $1$.  This indexing procedure may
look unusual but it makes easier to write the relations among
pfaffians.

Consider the action $\sigma\cdot(i_1\cdots
i_r)=\sigma(i_1)\cdots\sigma(i_r)$ of the symmetric group $S_n$ on the
set of rows. Given a row $I$ containing distinct entries let
$\tau_I\in S_n$ be the permutation reordering the entries of $I$ in
increasing order and fixing all integers not appearing in $I$ and let
$\gre(I)=(-1)^{\tau_I}$. We have $\pf_{\grs\cdot I} = \gre(I) \pf_I$
and $\pf_I=0$ if there are repetitions in $I$.  If
$I$ and $J$ are rows we set $IJ$ to be the row obtained by listing the
elements of $J$ after the elements of $I$. For a row $I=i_1i_2\cdots i_r$, let $|I|\doteq r$ be its length.

Let us define an order on rows as follows: $R=i_1i_2\cdots i_r\leq
S=j_1j_2\cdots j_s$ if $r\geq s$ and $i_h\leq j_h$ for
$h=1,2,\ldots,s$. For each pair of rows $R, S$ with increasing
entries which are not comparable with respect to this order we
construct a relation among pfaffians in the following way. Assume $R$
is not shorter than $S$ and let $R=IJ$ and $S=HK$ be such that $|I|=r$,
$|H|=r+1$, any entry of $I$ is less or equal to the corresponding
entry of $H$ and the first entry of $J$ is greater than the last entry
of $H$; so we have the first violation of the order condition in the $(r+1)$-th column.

\begin{theorem_}
For each pair of rows as above, we have the following relation among pfaffians:
\begin{equation}\label{equazione}
\sum(-1)^{\frac{h(h+1)}{2}}\gbin{|J\setminus K|+h}{h}{-1}
\gre(H'J')\gre(Z_2K'_1)\,\pf_{IJ'K'_1}\pf_{H'K'_2} = 0
\end{equation}
with $h\doteq|H|-|H'|$, $Z_1\doteq I\cap H$ and
$Z_2\doteq J\cap K$ as ordered rows and the sum running over
the set of all quadruples $(J',H',K'_1,K'_2)$ of rows with
increasing entries such that
\begin{itemize}
   \item[(1)] $|H'|  \leq   |H| $,
   \item[(2)] $ H'   \sqcup  J'   = H\sqcup J$,
   \item[(3)] $ K'_1 \sqcup  K'_2 = K$,
   \item[(4)] $ Z_1  \subset H' $,
   \item[(5)] $ Z_2  \subset J' \cap K'_2$.
\end{itemize}
Moreover these relations generate the ideal of relations
among pfaffians.
\end{theorem_}

In the formula above we have denoted, for $0\leq k\leq m$, by $\gbin{m}{k}{q}$ the \emph{gaussian binomials} defined as 
the element $\frac{(1-q^m)(1-q^{m-1})\cdots(1-q^{m-k+1})}{(1-q^k)(1-q^{k-1})\cdots(1-q)}$
of $\C[q]$ and $\gbin{m}{k}{-1}$ is this polynomial evaluated in $-1$.

\medskip

Our interest for this topic stems from the standard monomial
theory. Consider $V=\mC^{2n+2}$ equipped with a non degenerate
symmetric bilinear form such that the subspaces $V_1$ and $V_2$,
generated respectively by the first and the last $n+1$ vectors of the
canonical basis of $V$, are totally isotropic. The variety of the $(n+1)$--dimensional
totally isotropic subspaces of $V$ has two connected components; let
the \emph{positive lagrangian grassmannian} $\LGr$ be the component of those subspaces
whose intersection with
$V_2$ is even dimensional. The special orthogonal group of $V$ acts on this variety, let $G=\Spin(2n+2)$ be its simply
connected cover and $\spmo$ the Spin module of $G$.

Then $\LGr$ embeds into $\mP(\spmo)$ and the study of relations among pfaffians
is equivalent to the study of the equations defining the cone over this embedding of $\LGr$ (see section \ref{sectionPfaffians}). Notice that on the
grassmannian side we have a natural action of a bigger group of
symmetries which is not apparent on the pfaffian picture of the problem.
Moreover in the study of the coordinate ring of this embedding we can make use of the standard monomial theory.

The standard monomial theory is a very general theory which construct
a basis of the projective coordinate ring given by the immersion of a
generalized grassmannian or a flag variety. The prototypical example of
the theory is given by the work of Hodge on the Pl\"ucker immersion of
a grassmannian \cite{Hodge}. The idea of Hodge was generalized to a
projective immersion of partial flag varieties for a
semisimple group $G$ (and even to a Kac Moody one) between the
seventies and the nineties by the work of many peoples, Seshadri,
Lakshmibai, Musili, De Concini, Eisenbud, Procesi and Littelmann. The
results of standard monomial theory have many consequences both in
representation theory and in the study of singularities of Schubert
varieties and other related varieties. However the theory does not give
an explicit description of the equations beyond the original case of
the Pl\"ucker immersion of a grassmannian and some other very simple
cases.

In our case the standard monomial theory takes the following form. Let $\sfA$ be the coordinate ring of the embedding $\LGr\hookrightarrow\mP(\spmo)$ and let $\sfA_m$ be the subspace of homogeneous polynomials of degree $m$. There
exists a basis $x(I)$ of $\sfA_1$ indexed by increasing rows $I$ as above.
A \emph{tableau} $T=(I_1, I_2,\cdots, I_m)$ is simply a sequence of
rows; it is \emph{standard} if all rows have increasing entries and
$I_1\leq I_2\leq\cdots\leq I_m$. The set of standard monomials
$x(T)\doteq x(I_1)x(I_2)\cdots x(I_m)$ for $T$ standard are a basis for $\sfA_m$.  So
each section $x(T)$ with $T$ non standard may be written as a linear
combination of standard ones, these relations are called
\emph{straightening relations} for $\sfA$. Since the ideal of
relations in the generators $x(I)$ is generated by quadratic elements
it is enough to consider only non standard tableaux with two rows. An important point
is that in the straightening relation for the tableau $T=(I,J)$ only
standard tableaux $(H,K)$ with $H\leq I,\,J$ and $K\geq I,\,J$ do appear.

This is the core idea of the standard monomial theory: to replace the
knowledge of the (very) complicated explicit straightening relations
by that of the order condition stated above. Indeed this condition is
sufficient to deduce quite a lot about the geometry of the flag
variety $\LGr$.

The equations we write down for the generators $x(I)$ are given by
formula \eqref{equazione} replacing $\pf_I$ with
$x(I)$.  These are not straightening relations but they are what we
call \emph{shuffling relations}: given a non standard tableau $T$ with
two rows we say that the element $f=\sum a_h\, x(I_h)x(J_h)$, with
the $a_h\in\C$, is a shuffling relation for $T$ if $f=0$, $T$
appears with coefficient $1$ and all other tableaux do fulfill the
order condition of a straightening relation for $T$. In particular we are
\emph{not} asking that all tableaux but $T$ must be standard. It is
clear that by a finite number of steps we may deduce the straightening
relations from the shuffling relations; so these weaker relations are
still a set of generators for the ideal of relations. (The name
`shuffling' is the same name used for the analogous relations for the
determinants of minors of a matrix of indeterminates.) Notice that
also the classical Pl\"ucker equations are shuffling relations and not
straightening relations.

From the point of view of standard monomial theory, the easiest
cases, which were also the first ones to be analyzed \cite{Seshadri},
are those in which the embedding is in the projective space of a
minuscule $G$--module. Recall that an irreducible $G$--module is said
to be \emph{minuscule} if its weights are a single orbit under the Weyl group.

The condition for a module (for a general semisimple group) to be
minuscule is very strong and, so, very few modules are
minuscule. Moreover for all minuscule modules for classical groups but
the spin module explicit shuffling relations are known. Indeed all
such modules have an extremely simple order structure except the
fundamental representations of type $\sfA$ which corresponds to the
Pl\"ucker immersion of a grassmannian and the Spin modules. In particular our module $\spmo$ is a minuscule one, and the other are either twisted form of $\spmo$ or restriction of $\spmo$ to $\Spin(2n+1)$, so the same result for these other cases may be easily deduced from our result.

For the exceptional groups only two modules for $\sfE_6$ (one dual to the other) and one module
for $\sfE_7$ are minuscule. (For $\sfE_7$ one may see \cite{Equations} for
some hints to explicit formulas for straightening relations.)

\medskip

After completing this paper we came to known that formulas very
similar to ours have already been found by Kustin in
\cite{Kustin}. However the proof in that paper is very different from
our: while being much elementary, since it uses only multilinear
algebra, it does not exploit the role played by the representation
theory of the spin module. So we still think our proof of the
shuffling relations may have some interest; at least we hope this
paper may bring some attention to the paper of Kustin which, in our
opinion, is not very known.

The paper is organized as follows. In the first section we collect
some combinatorial definitions about rows, tableaux and standardness
and the definition and some properties of the gaussian binomials. The
second section introduce the spin group, the spin module and the
related grassmannian. Then the standard monomial theory for this
module is shortly discussed with a remark about a general invariance
of the relations defining a flag variety. In the third section we see
the definition of a pfaffian and the relations of such polynomial
functions with the spin module. Using this links we are able to prove
some invariance properties of the ideal of relations. In Section
\ref{sectionShufflingRelations} we introduce our formulas and prove
some combinatorial properties of such formulas. In
Section \ref{sectionRelations} we prove that the formulas are indeed shuffling relations.
Finally in the last section we extend these results to an arbitrary field.

We would like to thank Corrado De Concini, Peter Littelmann and Seshadri for useful conversations.

\section{Combinatorics}\myLabel{sectionCombinatorics}

\subsection{Rows and tableaux}

We call a finite sequence (maybe empty) of integers a \emph{row}. Let
$\rows{n}\supset\distintRows{n}\supset\standardRows{n}$ be
respectively the set of all rows in the alphabet $\{1,2,\ldots,n\}$,
the subset of rows containing distinct elements and the subset of
\emph{standard} rows, i.e. of rows $I=i_1i_2\cdots i_r$ with
$i_1<i_2<\cdots<i_r$. We define the lenght of the row $I=i_1i_2\cdots
i_r$ as $|I|\doteq r$, moreover the (standard) empty row is denoted by
$\emptyRow$. If $I \in \calR^+_n$ we define
$\widetilde{I}$ as the \emph{complementary} row $j_1\ldots j_{n-r}$ of
$I$ where
$\{i_1,\ldots,i_r\}\sqcup\{j_1,\ldots,j_{n-r}\}=\{1,2,\ldots,n\}$ and
$j_1<j_2<\ldots<j_{n-r}$.
If $I=i_1\dots i_r$ and $J= j_1\dots j_s$ are two rows we set $IJ
\doteq i_1\dots i_rj_1\dots j_s$.

Given an integer $h$ we define $I*h$ as the row obtained by adding $h$
to the end of the row $I$, futher if $k$ is another integer then
$I(h\replace k)$ is the row obtained by replacing all occurrences (if
any) of $h$ by $k$ in $I$. Hereafter we write $h\in I$ to say that
$h$ appears as an element in the row $I$ and, in general, we use the
language of sets with rows when this does not create any ambiguity.

The symmetric group $S_n$ acts on the set of rows: for $I=i_1i_2\cdots
i_r$ and $\sigma\in S_n$ we set $\sigma\cdot
I\doteq\sigma(i_1)\sigma(i_2)\cdots \sigma(i_r)$. Given a row
$I=i_1i_2\cdots i_r$ in $\distintRows{n}$ let $I^\leq$ be the row with
the same entries of $I$ rearranged in ascending order, moreover let
$\tau_I\in S_n$ be such that $\tau_I\cdot I=I^\leq$ and $\tau_I$ fixes
each integer $h\not\in I$, let also $\gre(I)$ be the sign of $\tau_I$.


We define the partial order for the standard rows: if $I=i_1\cdots
i_r,\,J=j_1\cdots j_s$ are standard rows we set $I\leq J$ if
$|I|\geq|J|$ and $i_m\leq j_m$ for all $1\leq m\leq|J|$.

A \emph{tableau} is a sequence of rows $T=(I_1, I_2, \cdots,
I_d)$ with $|I_1|\geq|I_2|\geq\cdots\geq|I_d|$; it is
standard if all rows are standard and $I_1\leq I_2\leq \cdots\leq
I_d$, the \emph{degree} of $T$ is the number $d$ of rows. If
$I_m=i_{m,1}i_{m,2}\cdots i_{m,r_m}$ for $m=1,\ldots,d$ then, as
customary, we write $T$ arranged vertically and left justified
$$
\left(
\begin{array}{cccccc}
i_{1,1} & i_{1,2} & \cdots & \cdots & \cdots & i_{1,r_1}\\
i_{2,1} & i_{2,2} & \cdots & \cdots & i_{2,r_2}\\
\vdots\\
i_{d,1} & i_{d,2} & \cdots & i_{d,r_d}\\
\end{array}
\right)
$$ so that $T$ is standard if and only if its entries increase along
the rows and do not decrease along the columns.

Finally let $\evenRows{n}$ be the subset of rows of even length of
$\rows{n}$ and define analogously $\evenDistintRows{n}$ and
$\evenStandardRows{n}$. Given a row $I\in\rows{n}$ let
$I^0\in\evenRows{n+1}$ be either the same row if $|I|$ is even or the
row $I*(n+1)$ otherwise; then the map $I\mapsto I^0$ is clearly a
bijection between $\standardRows{n}$ and $\evenStandardRows{n+1}$.

\subsection{Gaussian binomials}

Let $q$ be an indeterminate and define the \emph{gaussian binomials}
as the following elements of $\C[q]$: for all $m,k\in\Z$
$$
\gbin{m}{k}{q}\doteq\left\{
\begin{array}{ll}
\frac{(1-q^m)(1-q^{m-1})\cdots(1-q^{m-k+1})}{(1-q^k)(1-q^{k-1})\cdots(1-q)} & \textrm{if }0\leq k\leq m\\
0 & \textrm{otherwise.}
\end{array}
\right.
$$

The evaluation of $q$ to $-1$ is of particular interest for our aims;
the following result is easily proved by induction.
\begin{lemma}\myLabel{lemmaGaussianEvaluation}
For all $0\leq k\leq m$ we have
$$
\gbin{m}{k}{-1}=\left\{
\begin{array}{ll}
{\left\lfloor \frac{m}{2}\right\rfloor \choose \left\lfloor \frac{k}{2}\right\rfloor} & \textrm{if }m\textrm{ is odd or }k\textrm{ is even}\\
0 & \textrm{otherwise.}
\end{array}
\right.
$$
\end{lemma}
In the next Lemma we see a $q$--analogous of a well known binomial
formula; it is needed in the proof of our main result.
\begin{lemma}\myLabel{lemmaGaussianBinomials}
For all $1\leq s\leq m$ we have
$$
\sum_{h=0}^s (-1)^hq^{\frac{h(h-1)}{2}}\gbin{m}{s-h}{q}\gbin{m-s+h}{h}{q}=0.
$$
\end{lemma}
\begin{proof}
Let $\psi_{m,s}$ be the left hand side of the above identity. We use
induction on $m$. By the definition of the gaussian binomials we see
that for all $n\geq 0$ and $0\leq k\leq n+1$
$$
\gbin{n+1}{k}{q}=\frac{1-q^{n+1}}{1-q^{n+1-k}}\gbin{n}{k}{q}
$$
Hence for all $1\leq s\leq m$ we have
$$
\begin{array}{rcl}
\psi_{m+1,s} & = &
\sum_{h=0}^s(-1)^hq^{\frac{h(h-1)}{2}}\gbin{m+1}{s-h}{q}\gbin{m+1-s+h}{h}{q}\\ &
= &
\sum_{h=0}^s(-1)^hq^{\frac{h(h-1)}{2}}\frac{1-q^{m+1}}{1-q^{m+1-s+h}}\frac{1-q^{m+1-s+h}}{1-q^{m+1-s}}\gbin{m}{s-h}{q}\gbin{m-s+h}{h}{q}\\ &
= & \frac{1-q^{m+1}}{1-q^{m+1-s}}\psi_{m,s}\\ & = & 0\\
\end{array}
$$
by induction on $m$.

So we need to show that $\psi_{m,m}=0$ for all $m\geq 1$ to complete the proof. But
$$
\psi_{m,m}=\sum_{h=0}^m(-1)^hq^{\frac{h(h-1)}{2}}\gbin{m}{h}{q}
$$
and our claim follows by evaluating in $t=-1$ the Newton binomial formula
$$
\prod_{k=0}^{m-1}(1+q^kt)=\sum_{h=0}^mq^{\frac{h(h-1)}{2}}\gbin{m}{h}{q}t^h.
$$
\end{proof}

\begin{remark}\myLabel{remarkGaussianBinomial}
There is a certain link between gaussian binomials and Coxeter groups
of type $\sfA$. Recall that if $W=\langle s_1,s_2,\ldots,s_r\rangle$
is a Coxeter group and $I\subset\{s_1,s_2,\ldots,s_r\}$, then one may
define the set $W^I\hookrightarrow W$ of representatives of minimal
lenght of the quotient $W/W_I$, where $W_I$ is the subgroup generated
by $I$ in $W$. Moreover for any subset $S$ of $W$ the \emph{Poincar\'e
  polynomial} of $S$ is defined as $p_S(q)\doteq \sum_{\tau\in
  S}q^{\ell(\tau)}$, where $\ell:W\longrightarrow\N$ is the length
function.

We have
$p_{W^I}(q)=p_W(q)/p_{W_I}(q)$, the quotient of the Poincar\'e
polynomials of the two Coxeter groups $W$ and $W_I$. Moreover
$p_W(q)=\prod_i\frac{1-q^{d_i}}{1-q}$ where $d_1,d_2,\ldots$ are the
degrees of $W$ (see \cite{Humphreys}).

Now let $W$ be the symmetric group $S_m$ on $m$ elements, which is a
Coxeter group of type $\sfA_{m-1}$ with respect to the set of
generators $s_i=(i,i+1)$ for $i=1,2,\ldots,m-1$. Since the degrees of
a Coxeter group of type $A_{m-1}$ are $2,3,\ldots,m$, it follows at
once that the gaussian binomial $\gbin{m}{k}{q}$ is the Poincar\'e
polynomial of the set of minimal representatives of the quotient
$S_m/S_k\times S_{m-k}$.
\end{remark}

\section{The spin module}\myLabel{sectionSpinModule}

We fix once and for all the notation $\gre_1, \dots,\gre_n$ for the
standard basis of $\mC^n$ and we denote by $E_{ij}$ the matrix
associated to the linear map sending $\gre_j$ to $\gre_i$ and all
other elements of this basis to zero.

\subsection{The spin group}
On $V\doteq\mC^{2n+2}$ fix the symmetric bilinear form whose
associated matrix is $\left(\begin{smallmatrix}0 & I \\ I &
  0\end{smallmatrix}\right)$. Let $\gog\doteq\gos\goo(2n+2)$ and
  $G\doteq\Spin(2n+2)$ the associated simply connected group. In this
  basis $\gog$ is the set of all matrices $\left(\begin{smallmatrix} A
    & B \\ C & -^tA\end{smallmatrix}\right)$ with $A,B,C$
    $(n+1)\times(n+1)$ matrices and $B,\,C$ antisymmetric.

We fix the Chevalley generators, a Cartan subalgebra $\got$ and a
Borel subalgebra $\gob$ in the following way:
\begin{align*}
 e_i &\doteq \left(\begin{matrix} E_{i,i+1} & 0 \\ 0 &
   -E_{i+1,i}\end{matrix}\right) \mfor i=1,\dots,n,\\ e_{n+1} &\doteq
 \left(\begin{matrix} 0 & E_{n,n+1}-E_{n+1,n} \\ 0 &
   0\end{matrix}\right),\\ f_i &\doteq e_i^t\text{ for
   }i=1,\ldots,n+1,\\ h_i &\doteq \left(\begin{matrix} E_{i,i}-
     E_{i+1,i+1} & 0 \\ 0 &
     -E_{i,i}+E_{i+1,i+1}\end{matrix}\right)\mfor
   i=1,\dots,n,\\ h_{n+1} &\doteq \left(\begin{matrix}
     E_{n,n}+E_{n+1,n+1} & 0 \\ 0 &
     -E_{n,n}-E_{n+1,n+1}\end{matrix}\right), \\ \got &\doteq
   \{\left(\begin{matrix} A & 0\\ 0 & -A^t\end{matrix}\right)\in \gog
     \,|\, A \text{ is diagonal}\},\\ \gob
     &\doteq\{\left(\begin{matrix} A & B\\ 0 &
       -A^t\end{matrix}\right)\in \gog \,|\, A \text{ is upper
         triangular and }B\text{ is antisymmetric }\}.
\end{align*}
We denote with $T$ (resp. $B$) the maximal torus (resp. Borel subgroup) of $G$ whose Lie algebra is given by $\got$ (resp. $\gob$)and let $B^-$ be the Borel opposite to $B$.

We identify the Cartan subalgebra with $\mC^{n+1}$ mapping $\gre_i$ to $\left(\begin{smallmatrix} E_{i,i} & 0 \\ 0 &
-E_{i,i}\end{smallmatrix}\right)$ for $i=1,\ldots,n+1$ and, further, we identify the Cartan subalgebra with its dual using the standard form $\gre_i\cdot\gre_j=\grd_{ij}$ for $1\leq i,j\leq n+1$. In particular if we set $\gra_i \doteq h_i$, for $i=1,\ldots,n+1$, the set $\{\gra_1,\dots,\gra_{n+1}\}$ is a simple basis for the roots of $\gog$.
Let $\Lambda \doteq \Hom(T,\mC^*)\subset \got^*$ be the set of integral weights and $\Lambda^+$ be the subset of $\Lambda$ of dominant weights and let $\omega_1,\dots,\omega_{n+1}$ be the fundamental
weights. For $\grl\in \Lambda^+$ let $V_\grl$ be the irreducible
representation of $G$ of highest weight $\lambda$.

\subsection{The spin module}

We want now to give an explicit description of the dual of the irreducible module of highest weight $\omega_{n+1}$. This is called the \emph{positive spin module} and will be denoted with $\spmo$.

Define the vector space $\spmo\doteq\oplus_I\mC\cdot
x(I)$ with basis elements $x(I)$, $I\in\standardRows{n}$. Next we
define the \emph{weight} of a row $I\in\distintRows{n}$ as
$$
\wt(I)\doteq\tfrac{1}{2} (\sum_{i\in I^0} \gre_i - \sum_{i\notin I^0}\gre_i)\in\Lambda.
$$
Now $V$ has a $\gog$--module structure defined as follows:
$$
e_i(x(I)) = \left\{
\begin{array}{ll}
x(I(i+1\replace i)) & \textrm{if }i+1\in I,\,i\not\in I,\\
0 & \textrm{otherwise}
\end{array}
\right.
$$
$$
e_n(x(I)) = \left\{
\begin{array}{ll}
x(I*n) & \textrm{if }n\not\in I\textrm{ and }|I|\textrm{ is odd}\\
0 & \textrm{otherwise}
\end{array}
\right.
$$
$$
e_{n+1}(x(I)) = \left\{
\begin{array}{ll}
x(I*n) & \textrm{if }n\not\in I\textrm{ and }|I|\textrm{ is even}\\
0 & \textrm{otherwise.}
\end{array}
\right.\\
$$
$$
h_i(x(I))=\wt(I)x(I)\textrm{ for }i=1,2,\ldots,n+1.
$$
$$
f_i(x(I)) = \left\{
\begin{array}{ll}
x(I(i\replace i+1)) & \textrm{if }i\in I,\,i+1\not\in I,\\
0 & \textrm{otherwise}
\end{array}
\right.
$$
$$
f_n(x(I)) = \left\{
\begin{array}{ll}
x(I\setminus\{n\}) & \textrm{if }n\in I\textrm{ and }|I|\textrm{ is even}\\
0 & \textrm{otherwise}
\end{array}
\right.
$$
$$
f_{n+1}(x(I)) = \left\{
\begin{array}{ll}
x(I\setminus\{n\}) & \textrm{if }n\in I\textrm{ and }|I|\textrm{ is odd}\\
0 & \textrm{otherwise.}
\end{array}
\right.\\
$$
Here is an example of this action for $n=3$: for each basis vector of $\spmo$ we have drawn all operators $e_1,\,e_2,\,e_3$ which does not send that vector to $0$.
$$
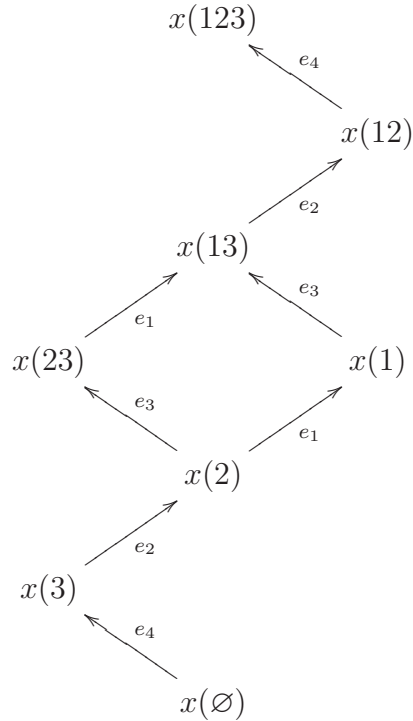
\xymatrix{
       & x(123) \\
       &       & x(12)\ar[lu]_{e_4} \\
       & x(13)\ar[ru]_{e_2} \\
 x(23)\ar[ru]_{e_1} &       & x(1)\ar[lu]_{e_3} \\
       & x(2)\ar[lu]_{e_3}\ar[ru]_{e_1} \\
 x(3)\ar[ru]_{e_2} \\
       & x(\emptyRow)\ar[lu]_{e_4} \\
}
$$
Notice that $\spmo$ is irreducible, and its lowest weight is $-\omega_{n+1}$ so it is the dual of
$V_{\omega_{n+1}}$. Notice also that all the weights are in the orbit of $-\omega_{n+1}$, in particular the spin module is miniscule.

\subsection{The lagrangian grassmannian}

Let $V_1$ and $V_2$ be the span, respectively, of the first, and the last, $n+1$ vector of the canonical basis of $V$. Define the positive lagrangian grassmannian $\LGr$ as the variety of $(n+1)$--dimensional subspaces of $V$ which have even dimensional intersection with $V_2$. This is an homogeneous space for the special orthogonal group of $V$ and for $G$.

The Picard group of $\LGr$ has a unique ample generator that we denote by
$\calL$ which is $G$--linearizable and the
$G$--module $H^0(\LGr,\calL) = \{\eta:G\lra \C \,|\, \eta\textrm{
  is holomorphic and }\eta(gp^{-1}) =
\omega^{-1}_{n+1}(p)\eta(g)\textrm{ for all }g\in G\textrm{ and }p\in
B\}$ is isomorphic to $\spmo$.

The main object of study of this paper is the graded ring $\sfA \doteq
\bigoplus_m H^0(\LGr,\calL^m)$. For $m\geq 0$ let $\sfA_m \doteq
H^0(\LGr,\calL^m)$ be the space of its homogeneous components of degree
$m$. In particular $A_1=V_{\omega_{n+1}}^*$. The ring $\sfA$ is
generated in degree one with quadratic relations. We denote by
$\calK\subset S^2(\sfA_1) \lra \sfA_2$ the kernel of the
multiplication map; our aim here is to find explicitly generators for
this kernel. When we need to stress the rank $n+1$ of the spin group
we add a subscript $n$, for example $\calK_n$.

\subsection{Standard monomial theory}

On the module $\spmo$ we have defined the basis $x(I)$, $I\in\standardRows{n}$, now we
want to extend the symbol $x(I)$ to any row $I\in\rows{n}$: let
$x(I)\doteq(-1)^{\tau_I}x(I^\leq)$ if $I\in\distintRows{n}$ and
$x(I)\doteq0$ for all $I\not\in\distintRows{n}$.

Next we extend $x:\rows{n}\longrightarrow\sfA_1$ to tableaux as
$$
x\left(
\begin{array}{c}
I_1\\
I_2\\
\vdots\\
I_r\\
\end{array}
\right) \doteq x(I_1)x(I_2)\cdots x(I_r)\in\sfS^r\sfA_1
$$ We will call such a monomial \emph{standard} if the tableau is
standard. Notice that if this monomial is non zero (i.e. if and only
if $I_h\in\distintRows{n}$ for all $h=1,\ldots,r$) then we may always
consider all of its rows as standard up to a sign change.

Let $y(T)$ be the image of $x(T)$ in $\sfA_r$ and recall that, by a
well known result of Standard Monomial Theory, the set of monomials
$y(T)$ with $T$ a standard tableau of degree $r$ is a $\C$--basis of
$\sfA_r$. Moreover for each non standard tableau $\binom{I}{J}$ we
have a \emph{straightening relation}
$$
y\binom{I}{J}=\sum a_{H,K}y\binom{H}{K},\quad a_{H,K}\in\C
$$ where the sum runs over all standard tableaux $\binom{H}{K}$ with
$H\leq I,\,J$ and $K\geq I,\,J$ (see Corollary $1$ of \cite{Seshadri}). In the case of a minuscule module these properties are not difficult to prove and were the starting point of the standard monomial theory.

Notice
that for all $\binom{H}{K}$ we have $I\cup J=H\cup K$ with multiplicities, by $\got$--homogeneity. We may write the relation above also as $x\binom{I}{J}-\sum a_{H,K}x\binom{H}{K}\in\calK$; in
particular these elements (that we still call straithening relations)
generate the space $\calK$ of quadratic relations, hence they generate the ideal defining the ring $\sfA$.

Later we will see
direct formulas for what we call \emph{shuffling relation}: given a
non standard tableau $T=\binom{I}{J}$ as above, a shuffling relation
for $T$ is any element
$$
\sum_m a_m x\binom{H_m}{K_m},\quad a_m\in\C
$$
of $\calK$ such that $x\binom{I}{J}$ appears with coefficient $1$
and, for all $m$, $H_m\leq I,\, J$ and $K_m\geq I,\,J$. Notice that we
are \emph{not} requiring $\binom{H_m}{K_m}$ to be standard. Clearly
the straightening relation for $\binom{I}{J}$ can be obtained in a
finite number of steps using the shuffling relations, so also the shuffling relations generates $\calK$.

\subsection{Complementary invariance of relations}

In this section we want to prove a symmetry property of the ideal
defining a (general) flag variety, so here we allow $\lieg$ to be any
semisimple Lie algebra.

Fix a triangular decomposition $\gog=\gob^-\oplus\goh\oplus\gob$, a
corresponding set of simple roots $\alpha_1,\ldots,\alpha_\ell$ with
fundamental weights $\omega_1,\ldots,\omega_\ell$ and corresponding
Chevalley generators $e_1,\ldots,e_\ell$, $h_1,\ldots,h_\ell$ and
$f_1,\ldots,f_\ell$. For any dominant weight $\lambda$, let
$v_\lambda$ be a fixed highest weight vector of the irreducible
$\gog$--module $V_\lambda$ and let $v_\lambda^-$ be the lowest weight
vector in the Weyl group orbit of $v_\lambda$. Denote by
$\calK_\lambda$ the kernel of the $\gog$--module projection
$\sfS^2V_\lambda\longrightarrow V_{2\lambda}$; this kernel is the
direct sum of all isotypic components of $\sfS^2V_\lambda$ of weight
less than $2\lambda$.

Let $w_0$ be the longest element of the Weyl group of $\gog$ and
denote by $d$ the linear map of $\Lambda$ given by $-w_0$; hence, in
particular, $V_\lambda^*\simeq V_{d(\lambda)}$ for all dominant
weights $\lambda$. We still denote by $d$ the permutation of
$\{1,\ldots,\ell\}$ defined in the following way: $1\leq h \leq\ell$
is mapped to $k$ if $d(\omega_h)=\omega_k$. Further let
$a:\gog\longrightarrow\gog$ be the unique automorphism defined by
$e_i\longmapsto f_{d(i)}$, $h_i\longmapsto -h_{d(i)}$ and
$f_i\longmapsto e_{d(i)}$ for $i=1,\ldots,\ell$ and, finally, extend
it to the universal enveloping algebra $U(\gog)$.

It is clear that any element of $\sfS^2 V_\lambda$ may be written as
$\sum_{h=1}^r\varphi_h v_\lambda\cdot\psi_h v_\lambda$ where
$\varphi_h,\,\psi_h$ are in the universal enveloping algebra
$U(\gob^-)$ of $\gob^-$ for $h=1,\dots,r$.
\begin{lemma}\myLabel{lemmaComplementaryInvariance}
If the element $\sum_{h=1}^r\varphi_h v_\lambda\cdot\psi_h v_\lambda$
of $\sfS^2 V_\lambda$ is in $\calK_\lambda$ then also $\sum_{h=1}^r
a(\varphi_h)v^-_\lambda\cdot a(\psi_h)v^-_\lambda$ is an element of
$\calK_\lambda$.
\end{lemma}
\begin{proof}
Given a $\gog$--module $V$, let $V^a$ be the $\gog$--module on the
vector space $V$ with action defined by $x\cdot v = a(x)v$ for all
$x\in\gog$ and $v\in V$. It is clear that $(\sfS^2 V)^a = \sfS^2 V^a$
and $(U\oplus V)^a=U^a\oplus V^a$ for all $\gog$--modules $U$ and $V$.

Notice that $v_\lambda^-$ is a highest weight vector of
$V_\lambda^a$, and its weight is
$a(w_0\lambda)=-dw_0(\lambda)=\lambda$; hence there exists a
$\gog$--module isomorphism $\alpha_\lambda:V_\lambda\longrightarrow
V_\lambda^a$ and we may normalize it by
$\alpha_\lambda(v_\lambda)=v_\lambda^-$.

Now consider the map $\sfS^2\alpha_\lambda$. Since
$V_{2\lambda}\subset\sfS^2 V_\lambda$ is sent to itself by
$\sfS^2\alpha_\lambda$, we see that
$\sfS^2\alpha_\lambda(\calK_\lambda)\subset\calK_\lambda$. So
$$
\begin{array}{rcl}
\sfS^2\alpha_\lambda(\sum_{h=1}^r\varphi_h v_\lambda\cdot\psi_h v_\lambda) & = & \sum_{h=1}^r\alpha_\lambda(\varphi_h v_\lambda)\cdot\alpha_\lambda(\psi_h v_\lambda)\\
 & = & \sum_{h=1}^r a(\varphi_h) v_\lambda^-\cdot a(\psi_h) v_\lambda^-\\
\end{array}
$$
is an element of $\calK_\lambda$ as claimed.
\end{proof}

Specializing to our context, and using the notation of previous
sections, we see that if $\varphi v_\lambda$, with $\varphi\in
U(\gob^-)$, is the vector $x(I)$, with $I\in\standardRows{n}$, then
$a(\varphi)v_\lambda^-$ is the vector $x(\widetilde{I})$ where
$\widetilde{I}$ is the complementary row of the row $I$. Hence we have
proved the following corollary.
\begin{corollary}\myLabel{corollaryComplementaryRows}
If the element $\sum_h a_h x\binom{I_h}{J_h}$, with $a_h\in\C$ and
$I_h,\,J_h\in\standardRows{n}$, is in the kernel $\calK$ of the
multiplication map $\sfS^2\sfA_1\longrightarrow\sfA_2$ then also
$\sum_h a_h x\binom{\widetilde{I_h}}{\widetilde{J_h}}$ is an element
of $\calK$.
\end{corollary}

\section{Pfaffians}\myLabel{sectionPfaffians}

\subsection{Definition}

Let $\goX_n$ be the set of $n\times n$ antisymmetric matrices with complex coefficients. We identify $\goX_n$ with $\Lambda^2 \C^n$ mapping $X = (x_{ij})\in \goX_{n}$ to $\omega_X \doteq \tfrac{1}{2}\sum_{i,j}x_{ij}\,\gre_i \wedge\gre_j$. We recall that the \emph{pfaffian} of $X\in\goX_{2m}$, denoted by $\Pf(X)$, is defined by
$$
\frac{1}{m!}\, \omega_X^m = \Pf(X)\, \gre_1\wedge \dots \wedge\gre_{2m}.
$$
Notice that $\Pf(X)^2=\det X$ and, in particular, the pfaffian of $X$ vanishes when the matrix $X$ is singular. Moreover if we let $S_{2m}$ act on $\goX_{2m}$ by permuting the rows and the coloumns of a matrix as $\sigma\cdot (x_{i,j})=(x_{\sigma(i),\sigma(j)})$, then $\Pf(\sigma\cdot X)=(-1)^\sigma\Pf(X)$.

Now we see an example; it will be used in the proof of our theorem.

\begin{example}\myLabel{examplePfaffians} The pfaffian of the following antisymmetric matrix
$$
\left(
\begin{array}{cccccc}
 0  &  1 & 0 &   & \ldots & 0\\
-1  &  0 & 1 &   & \ldots & 0\\
 0  & -1 & 0 & 1 & \ldots & 0\\
    &    &   & \vdots\\
    &        &   & -1 &  0 & 1\\
 0  & \ldots &   &    & -1 & 0\\
\end{array}
\right)
$$
is $1$; clearly the same is true for all even dimensional principal submatrices of $M$.
\end{example}

Let $\sfB$ be the ring of polynomial functions on $\goX_{n+1}$. For each $X \in \goX_{n+1}$ and $I=i_1\cdots i_h \in \calR_{n+1}$ define $X_I$ to be the antisymmetric matrix given by the principal minor corresponding to the rows and the columns $i_1,\dots,i_h$ of $X$. In this definition we allow repetitions and we consider also the order of the elements in $I$; for example if
$$X =\left(
  \begin{smallmatrix}
     0 &  a & b \\
    -a &  0 & c \\
    -b & -c & 0
  \end{smallmatrix}
\right) \quad \text{ then } \quad X_{13} = \left(
  \begin{smallmatrix}
    0 & b \\
    -b &  0
  \end{smallmatrix}
\right) \quad \text{ while } \quad X_{31} = \left(
  \begin{smallmatrix}
     0 & -b \\
    b & 0
  \end{smallmatrix}
\right).
$$
For each $I \in \calR_n$ we define a (polynomial) function $\pf_I$ on $\goX_{n+1}$ by
$$
\pf_I(X) \doteq \Pf(X_{I^0}).
$$
The ordering of the elements of $I$ is not essential here since $\pf_I(X) = (-1)^{\tau_I}\pf_{I^\leq}(X)$ for all $I\in\distintRows{n}$; however it will be convenient for us to have this more general notation. Notice that $\pf_I=0$ for all sequences $I$ with a repetition since the pfaffian of a singular matrix is zero. It is also clear that the functions $\pf_I$ with $I\in\rows{n}$ generate the ring $\sfB$; indeed $\pf_{ij}(X)=x_{ij}$ for all $1\leq i<j\leq n$.

\subsection{Pfaffians and the spin module}\label{section32}
We now describe the relation between the pfaffians and the spin module. Notice that $\LGr$ has a unique $B^-$--stable divisor that we will denote with $\calS$. This is the variety of all subspaces $W\in \LGr$ with non trivial intersection with $V_2$. As a subvariety of $\LGr$ it is defined by the equation $x(\emptyRow)=0$.

Let now $\calp\in\LGr$ be the point corresponding to $V_1$. The $B^-$--orbit $\calU$ of $\calp$ is the complement of $\calS$. More precisely it is an orbit under the unipotent radical of the stabilizer $P$ of $\calp$. Define $u:\goX_{n+1}\lra G$ by $u(X)\doteq\left(\begin{smallmatrix} I & 0\\ X & I\end{smallmatrix}\right)$ and $\grf: \goX_{n+1} \lra \LGr$ by $\grf(X) \doteq u(X) \cdot\calp$. Then $\grf$ is an isomorphism between $\goX_{n+1}$ and $\calU$.

The pull back $\grf^*\calL$ of $\calL$ on $\goX_{n+1}$ is isomorphic to the trivial line bundle so it induces a ring homomorphism $\psi:\sfA\lra \sfB$. Notice that $\psi(x(\emptyRow))$ is a nowhere vanishing function so it is a non zero constant that we can normalize to be $1$. Hence $\psi$ induces an isomorphism
$$ \bar \psi: \frac{\sfA}{(x(\emptyRow)=1)} \isocan \sfB. $$
In particular, since $\calU$ is open in $\LGr$, the restriction $\psi_m = \psi|_{\sfA_m}$ to the homogenoeus component is injective and we define $\sfB_m\doteq\psi(\sfA_m)$. Since $\psi(x(\emptyRow)) = 1$ we have $\sfB_0\subset \sfB_1 \subset \sfB_2 \subset \cdots$. Notice also that the isomorphism $\psi$ does not define a $G$--structure on $\sfB$ however it defines a $G$--structure on $\sfB_m$ such that the multiplication maps $\sfB_m \times \sfB_{m'}$ in $\sfB_{m+m'}$. In particular we have the following commutative diagram
$$
\begin{CD}
  \sfA_1 \otimes \sfA_1 @>>> \sfA_2 \\
   @V{\psi_1\otimes\psi_1}VV @V{\psi_2}VV \\
  \sfB_1 \otimes \sfB_1 @>>> \sfB_2
\end{CD}
$$
where vertical maps are isomorphisms and horizontal ones are induced by multiplication. Hence the homogenous relations between elements of $\sfA_1$ and $\sfB_1$ are the same.

We want to identify in a more explicit way $\sfB_1$; in order to do this we need to make explicit the trivialization of the line bundle $\calL$. If  $\grs : G \lra \C$ is a meromorphic function such that  $\grs(gp^{-1}) = \omega_{n+1}^{-1}(p)\grs(g)$ for all $p\in P$ then it defines naturally a meromorphic section of $\calL$. We can associate to $\grs$ also a function on $\goX_{n+1}$ by $f_\grs(X)\doteq\grs(u(X))$.

On the other hand every meromorphic function on $\goX_{n+1}$ can be constructed in this way. So we obtain an action of the Lie algebra $\gog$ on the space of meromorphic functions on $\goX_{n+1}$. Explicitely this action
is given as follows:
\begin{align*}
  (\left(\begin{array}{cc}A & 0\\0 & -^tA\end{array}\right) \cdot f) (X) &= \frac{\mtd\;}{\mtd s} f(X+s (A^t\,X+X\,A))
      \big|_{s=0} - \tfrac 12 \Tr (A) f(X) \\
  (\left(\begin{array}{cc}0 & B\\0 & 0\end{array}\right)\cdot f) (X) &= \frac{\mtd\;}{\mtd s} f(X+s XBX)
      \big|_{s=0} - \tfrac 12 \Tr (BX) f(X) \\
  (\left(\begin{array}{cc}0 & 0\\C & 0\end{array}\right)\cdot f) (X) &= \frac{\mtd\;}{\mtd s} f(X-sC) \big|_{s=0} \\
\end{align*}
for all $A\in \End(\C^{n+1})$ and $B,C\in \goX_{n+1}$.

In particular with simple computations we get that if $I\in \standardRows{n}$:
\begin{align*}
 e_i (\pf_I) &=
 \begin{cases}
   \pf_{e_i(I)} &\mif e_i(I)\neq 0,\\
   0 &\text{ otherwise};
 \end{cases}\\
h_i(\pf_I)&=\langle h_i,\wt(I)\rangle\pf_I.
\end{align*}
Since we have normalized $\psi$ in such a way that $\psi(x(\emptyRow))=1=\pf_\emptyRow$ we have
$$
\psi(\pf_I)=x(I)
$$
for all $I\in\standardRows{n}$. So, by our conventions, $\psi(\pf_I)=x(I)$ for all rows $I$ and not only for standard rows. Notice that, in particular, $\sfB_1$ is the vector space spanned by the functions $\pf_I$ with $I\in \standardRows{n}$.

So we have proved the following result.
\begin{proposition}\myLabel{propositionPfaffianWeightVector}
The element $\sum_h a_h x\binom{I_h}{J_h}$ of $\sfS^2 A_1$ is in $\calK$ if and only if $\sum_h a_h\pf_{I_h}\pf_{J_h}=0$ as a polynomial function on $\goX_{n+1}$.
\end{proposition}

\subsection{The action of the symmetric group and submatrices}

The previous identification of the weight vectors of the positive spin module and the pfaffians allows us to prove the following invariance properties. These will be used in the proof of our main result.

\begin{lemma}\myLabel{lemmaPermutationInvariance}
If $\sum_h a_h x\binom{I_h}{J_h}$ is an element of $\calK$, with $a_h\in\C$, and $\sigma\in S_n$ then $\sum_h a_h x\binom{\sigma\cdot I_h}{\sigma\cdot J_h}$ is in $\calK$ too.
\end{lemma}
\begin{proof}
Indeed $\sum a_h x\binom{I_h}{J_h}$ is an element of $\calK$ if and only if $\sum a_h\pf_{I_h}\pf_{J_h}=0$ by Proposition \ref{propositionPfaffianWeightVector}. If we change the enumeration of the rows and coloumns of an antisymmetric $(n+1)\times(n+1)$ matrix from $1,2,\ldots,n+1$ to $\sigma(1),\sigma(2),\ldots,\sigma(n),n+1$ the same relation holds. But in terms of the original enumeration this relation is $\sum a_h\pf_{\sigma\cdot I_h}\pf_{\sigma\cdot J_h}=0$; hence our claim follows using again Proposition \ref{propositionPfaffianWeightVector}.
\end{proof}

\begin{lemma}\myLabel{lemmaRestrictionInvariance}
If $f\doteq\sum_h a_h x\binom{I_h}{J_h}$ is an element of $\calK_n$, with $a_h\in\C$, and $I_h,J_h\subset\{1,\ldots,n-1\}$ for all $h$, then we may consider $f$ as an element of $\calK_{n-1}$. On the converse any relation in $\calK_{n-1}$ may be considered as an element of $\calK_n$.
\end{lemma}
\begin{proof}
The element $f\in\calK_n$ corresponds to the relation $\sum a_h\pf_{I_h}\pf_{J_h}=0$ in terms of pfaffians. These pfaffians involve the rows and coloumns $1,2,\ldots,n-1,n+1$ by hypothesis. Now we change the enumeration from $1,2,\ldots,n-1,n,n+1$ to $1,2,\ldots,n-1,n+1,n$. The same relation is still true with the new enumeration; but it is obtained by completing each odd length row $I$ by adding $n$ and not $n+1$ before computing the pfaffian $\pf_I$. Hence this last relation is the same we obtain if we consider $f$ has an element of the ring $\sfA$ for $n-1$ instead of $n$.

The second claim is analogously proved using again Proposition \ref{propositionPfaffianWeightVector}.
\end{proof}

If $I=i_1\ldots i_r$ is a row and $1\leq s\leq n$ an integer, let $j_s(I)$ be the row $i_1\cdots i_h si_{h+1}\cdots i_r$ where $h$ is the maximal index such that $i_t\leq s$ for all $t=1,\ldots,h$. Notice that $j_s(I)$ is standard if $I$ is standard and $s\not\in I$.

Let $d_s(I)$ be the row obtained by deleting any occurrence of $s$ in $I$. We have that $d_s(I)$ is standard if $I$ is standard.

Further we define $j_s$, $d_s$ on weight vectors by $j_s(x(I))\doteq x(j_s(I))$, $d_s(x(I))\doteq x(d_s(I))$ respectively, for all standard rows $I$; clearly $j_sx(I)=0$ if and only if $s\in I$. We extend $j_s$ and $d_s$ from $\sfA_1$ to $\sfS^*\sfA$ as algebra homomorphisms.

\begin{lemma}\myLabel{lemmaInsertDeleteKernel}
Let $f\doteq\sum_h a_h x\binom{I_h}{J_h}$ be an element of $\calK$ with $I_h,\,J_h$ standard rows for all $h$. If $s\not\in I_h\cup J_h$ for all $h$, then $j_s(f)\in\calK$. If $s\in I_h\cap J_h$ for all $h$, then $d_s(f)\in\calK$.
\end{lemma}
\begin{proof}
Given a subset $\Delta$ of $\{1,2,\ldots,n\}$ and a standard row $I\subset\Delta$ let $c_\Delta(I)$ be the standard row complementary to $I$ in $\Delta$. We define $c_\Delta$ on weight vectors by $c_\Delta(x(I))\doteq x(c_\Delta(I))$ for all $x(I)$ such that $I$ is a standard row contained in $\Delta$; further we extend it as an algebra homomorphism from the subalgebra $\sfA_\Delta$ of $\sfA$ generated by $x(I)$ with $I\subset\Delta$ to $\sfA$.

By Corollary \ref{corollaryComplementaryRows}, Lemma \ref{lemmaPermutationInvariance} and Lemma \ref{lemmaRestrictionInvariance} we have: if $f\in\sfA_\Delta\cap\calK$ then $c_\Delta(f)\in\calK$.

Now let $\Delta_1\doteq\{1,2,\ldots,s-1,s+1,\ldots,n\}$, $\Delta_2\doteq\{1,2,\ldots,n\}$. If $I$ is a standard row and $s\not\in I$ then $j_s(I)=c_{\Delta_2}c_{\Delta_1}(I)$ and hence our first claim follows.

The second claim is analogous; indeed if $s\in I$ then $d_s(I)=c_{\Delta_1}c_{\Delta_2}(I)$.
\end{proof}

\section{Shuffling relations}\myLabel{sectionShufflingRelations}

\subsection{Shuffling polynomial}

We say that a tableau is $r$--\emph{standard} if all rows are standard, the entries of $T$ do not decrease along the first $r$ columns but the same is not true in the $(r+1)$--th column. If morever $T$ has two rows then we may write it as $T=\binom{IJ}{HK}$ where $I=i_1i_2\cdots i_r$, $J=j_1j_2\cdots j_s$, $H=h_1h_2\cdots h_{r+1}$ and $K=k_1 k_2\cdots k_t$ with $s>t$ and with the following inequalities
$$
T = \left(
\begin{array}{cccccccccccccccc}
i_1 & < & i_2 & < & \cdots & < & i_r & < & j_1 & < & j_2 & < & \cdots & \cdots & < & j_s\\
\verleq & & \verleq & & & & \verleq & & \vergreater \\
h_1 & < & h_2 & < & \cdots & < & h_r & < & h_{r+1} & < & k_1 & < & \cdots & < & k_t\\
\end{array}
\right)
$$
We call this subdivision of the two rows of $T$ its \emph{canonical} form. Notice in particular that we have a chain of strict inequalities from $h_1$ to $j_s$, hence there is no repetition in $HJ$.

However, despite the name, a standard tableau is not $r$--standard for any $r$.

\begin{definition}\myLabel{definitionShufflingPolynomial}
Given an $r$--standard tableau $T=\binom{IJ}{HK}$ in canonical form, we define the \emph{shuffling polynomial} of $T$ as the following element of $S^2\sfA_1$:
$$
\Theta(T) \doteq \sum(-1)^{\frac{h(h+1)}{2}}\gbin{|J\setminus K|+h}{h}{-1}\gre(H'J')\gre(Z_2K'_1)\,x\binom{IJ'K'_1}{H'K'_2}
$$
with $h\doteq|H|-|H'|$, $Z_1=Z_1(T)\doteq I\cap H$ and $Z_2=Z_2(T)\doteq J\cap K$ as ordered rows and the sum running over the set $\calI(T)$ of all $(J',H',K'_1,K'_2)\in\standardRows{n}^{\times 4}$ such that
\begin{enumerate}
\item[(1)] $|H'|\leq |H|$,
\item[(2)] $H'\sqcup J' = H\sqcup J$,
\item[(3)] $K'_1\sqcup K'_2 = K$,
\item[(4)] $Z_1\subset H'$,
\item[(5)] $Z_2\subset J'\cap K'_2$.
\end{enumerate}
For short let
$$
\alpha(J',H',K'_1,K'_2)\doteq(-1)^{\frac{h(h+1)}{2}}\gbin{|J\setminus K|+h}{h}{-1}\gre(H'J')\gre(Z_2K'_1).
$$
\end{definition}

\begin{remark}\myLabel{remarkShufflingRelation}
Notice that $T$ appears in $\Theta(T)$ with coefficient $1$ and each tableau in $\Theta(T)$ but $T$ fulfills the conditions of a straightening relation for $T$.
\end{remark}

\subsection{Combinatorial properties of shuffling polynomials}

In the next purely combinatorial lemma we prove that the shuffling polynomial of a tableau does not change if we insert or remove a common entry in the rows.

\begin{lemma}\myLabel{lemmaInsertDeleteTheta} Let $T$ be a non standard tableau with two standard rows.
\begin{enumerate}
\item If $T$ does not contain $s$, then $j_s(T)$ is not standard and $j_s\big(\Theta(T)\big) = \Theta\big(j_s(T)\big)$.
\item If $s$ is contained in both rows of $T$, then $d_s(T)$ is not standard and $d_s(\Theta(T))=\Theta(d_s(T))$.
\end{enumerate}
\end{lemma}
\begin{proof} We prove the first claim. Assume that $T=\binom{IJ}{HK}$ is $r$--standard in canonical form. Consider first the case $s<h_{r+1}$. Then $j_s(T)$ is $(r+1)$--standard, $j_s(T)= \binom {j_s(I)\, J}{j_s(H)\, K}$ in canonical form and the map
$$
(J',H',K_1',K_2') \mapsto (J',j_s(H'),K_1',K_2')
$$
gives a bijection from $\calI(T)$ to $\calI(j_s(T))$.

Let $f$ be an addend in $\Theta(T)$, we compare $j_s(f)$ with the corresponding term in $\Theta(j_s(T))$ under this bijection. Clearly we have
\begin{itemize}
\item[(i)] $|j_s(H)| - |j_s(H')| = |H| - |H'|$,
\item[(ii)] $J\senza K$ is the same set for $T$ and for $j_s(T)$ and
\item[(iii)] $Z_2$ and $K_1'$ do not change in this bijection.
\end{itemize}
Let $k$ be the number of elements of $J'$ smaller than $s$. Then we have
$$
\begin{array}{rcl}
\gre(j_s(H')J') & = & (-1)^k \gre(H'J') \\
\gre(j_s(I)J'K_1') & = & (-1)^k \gre(IJ'K_1') \\
\gre(j_s(H')K_2') & = & \gre(H'K_2')
\end{array}
$$
Hence
\begin{align*}
\gre(j_s(H')\, J') \, x\binom{j_s(I)J'K'_1}{j_s(H')K'_2}
& =
\gre(j_s(H')J')
\gre(j_s(I)J'K_1')
\gre(j_s(H')K_2')
\,
x\binom{(j_s(I)J'K'_1)^\leq}{(j_s(H')K'_2)^\leq} \\
& =
\gre(H'J')
\gre(IJ'K_1')
\gre(H'K_2')
x\binom{j_s((IJ'K'_1)^\leq)}{j_s((H'K'_2)^\leq)} \\
&=\gre(H'J')
\gre(IJ'K_1')
\gre(H'K_2')
j_s\bigg(x\binom{(IJ'K'_1)^\leq}{(H'K'_2)^\leq}\bigg) \\
&=\gre(H'J') j_s\bigg(x\binom{IJ'K'_1 }{ H'K'_2 }\bigg)
\end{align*}
which implies that the two terms we are considering are equal. This proves $j_s(\Theta(T)) = \Theta(j_s(T))$.

Assume now that $s>h_{r+1}$. Then $j_s(T)$ is $r$--standard, $j_s(T)= \binom {I\, j_s(J)}{H\, j_s(K)}$ in canonical form and the map
$$
(J',H',K_1',K_2') \mapsto (j_s(J'),H',K_1',j_s(K_2'))
$$
gives a bijection from $\calI(T)$ to $\calI(j_s(T))$. As above we compare the corresponding terms in $\Theta(T)$ and $\Theta(j_s(T))$ under this bijection. We have: $j_s(J)\senza j_s(K) = J\senza K$ while $|H|-|H'|$ does not change in this bijection. Let $k$ be the number of elements of $H'$ bigger than $s$ and $m$ be the number of elements of $K_1'$ smaller than $s$. Then we have
$$
\begin{array}{rcl}
\gre(H'j_s(J')) & = & (-1)^k \gre(H'J') \\
\gre(Ij_s(J')K_1') & = & (-1)^m \gre(IJ'K_1') \\
\gre(H'j_s(K_2')) & = & (-1)^k \gre(H'K_2') \\
\gre(j_s(Z_2) K_1') & = & (-1)^m \gre(Z_2 K_1') \\
\end{array}
$$
which, as above, implies that the two corresponding terms are equal and hence $j_s(\Theta(T)) = \Theta(j_s(T))$.

Now we see how the second claim follows from the first one. Indeed let $T$ be a non standard tableau containing $s$ in both rows as in the second claim and let $\overline{T}\doteq d_s(T)$. Notice that $j_s(\overline{T})=j_sd_s(T)=T$ since $T$ has standard rows, hence $\overline{T}$ can not be standard otherwise also $T=j_s(\overline{T})$ should be standard. So we may apply the first claim to $\overline{T}$ and find $j_s(\Theta(\overline{T}))=\Theta(j_s(\overline{T}))=\Theta(T)$. Apply $d_s$ to both sides of this identity and notice that $d_s j_s=\Id$ to conclude $\Theta(d_sT)=d_sj_s\Theta(\overline{T})=d_s\Theta(T)$ as claimed.
\end{proof}

Now we see another combinatorial property of the shuffling polynomials. We want to prove that if we permute the entries of a tableau with a permutation satisfing certain assumptions then the shuffling polynomial may change only by the sign.

Let us start by stating these assumptions of compatibilities between an $r$--standard tableau $T=\binom{IJ}{HK}$ in canonical form and a permutation $\sigma\in S_n$. Given a row $I$ let $I^\sigma$ be the row $(\sigma\cdot I)^\leq$ and let $T^\sigma\doteq\binom{(IJ)^\sigma}{(HK)^\sigma}$. We say that $\sigma$ is \emph{compatible} with $T$ if $T^\sigma$ is again $r$--standard with canonical form $\binom{I^\sigma J^\sigma}{H^\sigma K^\sigma}$ and $K^\sigma=\sigma\cdot K$. In particular notice that in this case we have $Z_1(T^\sigma)=Z_1(T)^\sigma$ and $Z_2(T^\sigma)=Z_2(T)^\sigma$.

\begin{lemma}\myLabel{lemmaSigmaTheta}
If $T$ is $r$--standard and $\sigma$ is $T$--compatible then $\sigma \cdot \Theta(T)=\pm\Theta(T^\sigma)$.
\end{lemma}
\begin{proof}
First notice that the map
$$
(J',H',K_1',K_2') \longmapsto (J'^\sigma,H'^\sigma,K_1'^\sigma,K_2'^\sigma)
$$
gives a bijection from $\calI(T)$ to $\calI(T^\sigma)$.

If $(H',J',K_1',K_2')$ is an element of $\calI(T)$, consider the following sequence of transformations which reorder the row $H'^\sigma J'^\sigma$
$$
\begin{CD}
H'^\sigma\;J'^\sigma@>{\tau_{\grs\cdot H'}^{-1},\tau_{\grs\cdot J'}^{-1}}>>
(\grs\cdot H')(\grs\cdot J') @>{\grs^{-1}}>>
H'\,J' @>{\tau_{H'\,J'}}>>
H\,J @>{\grs}>> \\
 @>{\grs}>>(\grs\cdot H)(\grs\cdot J) @>{\tau_{\grs\cdot H},\tau_{\grs\cdot J}}>>H^\sigma\,J^\sigma @>{=}>> (H'^\sigma\,J'^\sigma)^\leq.
\end{CD}
$$
So, using that the sign of $\sigma\tau_{H'\, J'}\sigma^{-1}$ is that of $\tau_{H'\, J'}$, we obtain
\begin{equation*}
  \gre(H'^\sigma\,J'^\sigma) = \gre(\grs\cdot H') \gre(\grs\cdot J') \gre(H'J')\gre(\grs\cdot H)\gre(\grs\cdot J).
\end{equation*}
Moreover by $\grs\cdot K = K^\sigma$ we have  $\grs\cdot K'_1 = K_1'^\sigma$, $\grs\cdot K'_2 = K_2'^\sigma$ and also $\gre(Z_2 K_1')=\gre(Z_2^\sigma K_1'^\sigma)$, hence
\begin{align*}
\gre(H'J')\gre(Z_2 K_1')\: \grs \cdot x\binom{IJ' K_1'}{H'K_2'} &= \gre(H'J')\gre(Z_2 K_1')\gre(\grs\cdot I)\gre(\grs\cdot J')\gre(\grs\cdot H')
x\binom{I^\sigma\;J'^\sigma\; K_1'^\sigma}{H'^\sigma\;K_2'^\sigma} \\
& =\gre(\grs\cdot I)\gre(\grs\cdot J)\gre(\grs\cdot H) \gre(H'^\sigma\;J'^\sigma)\gre(Z_2^\sigma K_1'^\sigma)
x\binom{I^\sigma\;J'^\sigma\; K_1'^\sigma}{H'^\sigma\;K_2'^\sigma}.
\end{align*}
Set $\gre \doteq\gre(\grs\cdot I)\gre(\grs\cdot J)\gre(\grs\cdot H)$ and notice that $|J'|=|J'^\sigma|$ and $|H| - |H'|=|H^\sigma| - |H'^\sigma|$. From the above identities we conclude
\begin{align*}
\grs \cdot \Theta (T)
& = \sum _{(H',J',K_1',K_2')\in \calI(T)} \gra(H',J',K_1',K_2') \; \grs \cdot x\binom{IJ' K_1'}{H'K_2'} \\
& = \sum _{(H',J',K_1',K_2')\in \calI(T)} \gre \; \gra(H'^\sigma,J'^\sigma,K_1'^\sigma,K_2'^\sigma) \;  x\binom{I^\sigma\;J'^\sigma\;K_1'^\sigma}{H'^\sigma K_2'^\sigma} \\
& = \gre \; \Theta(T^\sigma)
\end{align*}
proving our claim.
\end{proof}

\section{Proof of the shuffling relations}\myLabel{sectionRelations}

In this section we prove our main result.

\begin{theorem}\myLabel{theoremMain}
If $T$ is a non standard tableau with two standard rows, then $\Theta(T)$ is an element of $\calK$.
\end{theorem}
\begin{proof}
Assume that $T$ is $r$--standard, $T=\binom{IJ}{HK}$ in canonical form and let $R_1\doteq IJ$ and $R_2\doteq HK$ be its first and second row respectively.

{\tt Step 1.} If $R_1\cup R_2\neq\{1,2,\ldots,n\}$ then there exists $\sigma\in S_n$ such that $n\not\in\sigma\cdot T$ and $\sigma(i)<\sigma(j)$ for any pair $1\leq i<j\leq n$ in $T$. The permutation $\sigma$ is clearly $T$--compatible and, in particular $T^\sigma=\sigma\cdot T$, so $\Theta(\sigma\cdot T)=\pm\sigma\cdot\Theta(T)$ by Lemma \ref{lemmaSigmaTheta}.

Using induction on $n$ we may suppose that $\Theta(\sigma\cdot T)\in\calK_{n-1}$; but then $\Theta(\sigma\cdot T)\in\calK_n$ by Lemma \ref{lemmaRestrictionInvariance}. Hence $\Theta(T)=\pm\sigma^{-1}\cdot\Theta(\sigma\cdot T)\in\calK_n$ by Lemma \ref{lemmaPermutationInvariance} and our claim is proved.

The inductive base step is automatically true since for $n=1$ there are no non standard tableaux.

{\tt Step 2.} Now we proceed by induction on $|R_1\cap R_2|$. Suppose $s\in R_1\cap R_2$ and let $\overline{T}\doteq d_s(T)$. Since $s\not\in\overline{T}$ we have $\Theta(\overline{T})\in\calK$ by {\tt Step 1}. Hence $j_s(\Theta(\overline{T}))\in\calK$ by Lemma \ref{lemmaInsertDeleteKernel}. But $j_s(\Theta(\overline{T}))=\Theta(j_s(\overline{T}))$ by Lemma \ref{lemmaInsertDeleteTheta} and moreover $j_s(\overline{T})=T$ since $T$ has standard rows. This proves our claim.

In the following steps we assume that $R_1$ and $R_2$ do not intersect proving the induction base.

{\tt Step 3.} Now we show that it suffices to prove our claim for a particular tableau. Indeed let $I=i_1i_2\cdots i_r$, $J=j_1j_2\cdots j_s$, $H=h_1h_2\cdots h_{r+1}$, $K=k_1k_2\cdots k_t$, with $2r+s+t+1=n$ since we are assuming $R_1\cup R_2=\{1,\ldots,n\}$ and $R_1\cap R_2=\varnothing$.

By $R_1\cap R_2=\varnothing$ we deduce that there exists (a unique) $\sigma\in S_n$ such that $\sigma(i_u)=u$ for $u=1,\ldots,r$, $\sigma(h_u)=u+r$ for $u=1,\ldots,r+1$, $\sigma(j_u)=u+2r+1$ for $u=1,\ldots,s$ and $\sigma(k_u)=u+2r+s+1$ for $u=1,\ldots,t$. It is clear that
$$
T^\sigma=\sigma\cdot T=T^0\doteq\left(
\begin{array}{ccccccccc}
1 & 2 & \cdots & r & 2r+2 & 2r+3 & \cdots & \cdots & 2r+s+1\\
r+1 & r+2 & \cdots & 2r & 2r+1 & 2r+s+2 & \cdots & n\\
\end{array}
\right).
$$
In particular $T^\sigma$ is $r$--standard and $\sigma$ is $T$--compatible. Hence $\Theta(T)=\sigma^{-1}\cdot\Theta(T^0)$ by Lemma \ref{lemmaSigmaTheta} and we see that if we prove $\Theta(T^0)\in\calK$ then $\Theta(T)\in\calK$ by Lemma \ref{lemmaPermutationInvariance}.

So in the sequel we assume $T=T^0$.

{\tt Step 4.} In this step we prove our claim for $K=\varnothing$. For short we write $(J',H')$ instead of $(J',H',\varnothing,\varnothing)$.

Notice that in
$$
\Theta(T)=\sum_{(J',H')\in\calI(T)}\alpha(J',H')x\binom{IJ'}{H'}
$$
only $T$ is non standard. Let $\Delta(T)$ be the unique element of $\sfS^2\sfA_1$ corresponding to the straightening relation for $T$. Each tableau $\binom{R_1'}{R_2'}$ in $\Delta(T)$ verifies $R_1'\leq R_1,R_2$ and $R_2'\geq R_1,R_2$; using this it is easy to see that there exists $(J',H')\in\calI(T)$ such that $R_1'=IJ'$ and $R_2'=H'$. So we may write
$$
\Delta(T)=\sum_{(J',H')\in\calI(T)}\delta(J',H')x\binom{IJ'}{H'}
$$
for some coefficients $\delta(J',H')\in\C$. If we show that $\alpha(J',H')=\delta(J',H')$ for all $(J',H')\in\calI(T)$ then we have $\Theta(T)=\Delta(T)\in\calK$ and our claim is proved.

{\tt Step 4.1.} Now we want to prove that $\alpha(J',H')=\delta(J',H')$ for all $(J',H')\in\calI(T)$ with $r+1\in H'$. We compare $d_r e_r \Theta(T)$ with $d_r e_r \Delta(T)$. Let
$$
\overline{T}=\binom{d_r(I)J}{d_{r+1}(H)}=\left(
\begin{array}{cccccccc}
1 & 2 & \cdots & r-1 & 2r+2 & 2r+3 & \cdots & n\\
r+2 & r+3 & \cdots & 2r & 2r+1\\
\end{array}
\right).
$$
Using the definition of $d_r$ and $e_r$ we have
$$
d_re_r\Theta(T)=\sum_{(J',H')\in\calI(T),\,r+1\in H'}\alpha(J',H')x\binom{d_r(I)J'}{d_{r+1}(H')}.
$$
The map $(J',H')\longmapsto(J',d_{r+1}(H'))$ is a bijection from $\calI(T)$ and $\calI(\overline{T})$ (with inverse $(J',H')\longmapsto(J',j_{r+1}(H'))$). Moreover, adding the corresponding tableau as superscript for clarity, we have $$
\alpha^{\overline{T}}(J',d_{r+1}(H'))=(-1)^{\frac{\overline{h}(\overline{h}+1)}{2}}\gbin{n-2r-1+\overline{h}}{\overline{h}}{-1}\gre((H'\setminus r+1)J')
$$
with $\overline{h}=|d_{r+1}(H)|-|d_{r+1}(H')|=|H|-|H'|$ and also $\gre((H'\setminus r+1)J')=\gre(H'J')$ since $r+1\in H'J'$ is the minimum element. Hence $\alpha^{\overline{T}}(J',d_{r+1}(H'))=\alpha^T(J',H')$; this proves that $d_re_r\Theta(T)=\Theta(\overline{T})$.

But then $d_re_r\Theta(T)=\Theta(\overline{T})\in\calK$ by {\tt Step 1} since $r\not\in\overline{T}$. Notice that $\overline{T}$ is the unique non standard tableau in $\Theta(\overline{T})$, so $d_re_r\Theta(T)$ is the element $\Delta(\overline{T})$ of $\sfS^2\sfA_1$ corresponding to the straightening relation of $\overline{T}$ since it is an element of $\calK$ and the coefficient of $\overline{T}$ is $1$ in both elements.

Consider now $d_re_r\Delta(T)$. By the definitions
$$
d_re_r\Delta(T)=\sum_{(J',H')\in\calI(T),\,r+1\in H'}\delta(J',H')x\binom{d_r(I)J'}{d_{r+1}(H')}.
$$
Hence also in $d_re_r\Delta(T)$ the unique non standard tableau is $\overline{T}$ and it appears with coefficient $1$. But $\calK$ is closed by $e_r$ since it is a $\lieg$--module and it is closed by $d_r$ by Lemma \ref{lemmaInsertDeleteKernel} so $d_re_r\Delta(T)=\Delta(\overline{T})$ being the straightening relation for $\overline{T}$ unique.

Hence we have showed that $d_re_r\Theta(T)=\Delta(\overline{T})=d_re_r\Delta(T)$; in particular $\alpha(J',H')=\delta(J',H')$ for all $(J',H')\in\calI(T)$ with $r+1\in H'$ that is our claim for this step.

{\tt Step 4.2.} Our next claim is now $\alpha(J',H')=\delta(J',H')$ if $H'\not\neq\varnothing$. Indeed let $r+1\leq i\leq n$ and consider $e_i\Delta(T)$. We find at once
$$
\begin{array}{rcl}
e_i\Delta(T) & = & \sum_{(J',H')\in\calI(T), i\in H',i+1\in J'}\delta(J',H')x\binom{I,J'(i+1\longmapsto i)}{H'}+\\
 & & \sum_{(J',H')\in\calI(T), i+1\in H',i\in J'}\delta(J',H')x\binom{IJ'}{H'(i+1\longmapsto i)}.
\end{array}
$$
In particular it is easy to see that all tableaux in $e_i\Delta(T)$ are standard; but this is an element of $\calK$, hence $e_i\Delta(T)=0$. Each tableau in $e_i\Delta(T)$ is obtained in exactly two ways: replacing $i+1$ by $i$ in the first row or in the second row. So we have proved that $\delta(J'',H'')=-\delta(J',H')$ if $i$ and $i+1$ appear in different row in $\binom{IJ'}{H'}$ and $\binom{IJ''}{H''}$ is obtained from $\binom{IJ'}{H'}$ by swapping $i$ and $i+1$.

Since each tableau $\binom{IJ'}{H'}$ with $H'\neq\varnothing$ is reached by a finite number of swaps of $i,i+1$ with $r+1\leq i\leq n$ from a tableau containing $r+1$ in the second row, we have $\alpha(J',H')=\delta(J',H')$ for all tableaux with $H'\neq\varnothing$ using the result of the previous step.

{\tt Step 4.3.} So $\Theta(T)-\Delta(T)=c\cdot x\binom{IHJ}{\varnothing}$ and we want to show $c=0$. Let $M$ be the $(n+1)\times(n+1)$ antisymmetric matrix of Example \ref{examplePfaffians} whose all pfaffians are $1$ and identify the elements of $\sfS^*\sfA_1$ with the corrisponding polynomials as in the Proposition \ref{propositionPfaffianWeightVector}. So $c=\Theta(T)(M)$ since $\Delta(T)(M)=0$ by $\Delta(T)\in\calK$ and $x\binom{IHJ}{\varnothing}(M)=1$. We want to show that $\Theta(T)(M)=0$ proving $c=0$ and $\Theta(T)=\Delta(T)$.

We have
$$
\begin{array}{rcl}
\Theta(T)(M) & = & \sum_{(J',H')\in\calI(T)}(-1)^{\frac{h(h+1)}{2}}\gbin{n-2r-1+h}{h}{-1}\gre(H'J')\\
 & = & \sum_{h=0}^{|H|}(-1)^{\frac{h(h+1)}{2}}\gbin{n-2r-1+h}{h}{-1}\cdot\sum_{H'\subset HJ,\, |H'|=|H|-h}\gre(H'J')\\
\end{array}
$$
Notice that by the definition of $\gre$ on rows we have $\sum_{H'\subset HJ,\, |H'|=|H|-h}\gre(H'J')=\sum_\tau(-1)^\tau$, where the sum runs on the set of minimal rapresentatives of the quotient $S_{|H|+|J|}/S_{|H|-h}\times S_{|J|+h}=S_{n-r}/S_{r+1-h}\times S_{n-2r-1+h}$. Hence the previous sum is the Poincar\'e polynomial of this quotient evaluated in $-1$, so $\sum_{H'\subset HJ,\, |H'|=|H|-h}\gre(H'J')=\gbin{n-r}{r+1-h}{-1}$ by Remark \ref{remarkGaussianBinomial}. We find
$$
\Theta(T)(M) = \sum_{h=0}^{r+1}(-1)^{\frac{h(h+1)}{2}}\gbin{n-2r-1+h}{h}{-1}\gbin{n-r}{r+1-h}{-1}\\
$$
and this is zero by Lemma \ref{lemmaGaussianBinomials} with $s=|H|=r+1$, $m=|H|+|J|=n-r$ and $q=-1$.

This finishes the proof that $\Theta(T)\in\calK$ for all non standard tableau $T$ with $K=\varnothing$.

{\tt Step 5.} In this final step we prove that our claim for generic $K$ follows by the case $K=\varnothing$ of the previous steps. Let $T=\binom{IJ}{HK}$ be as in the conclusion of {\tt Step 3}; $I=12\cdots r$, $J=2r+2\cdots 2r+s+1$, $H=r\cdots 2r+1$ and $K=2r+s+2\cdots n$.

We want to proceed by induction on $|K|$. Indeed, for $u=0,\ldots,|K|$ let $T_u\doteq \binom{IJ}{H,K_u}$, where $K_u$ is the standard row containing the first $u$ entries of $K$, and notice that the base inductive step, i.e. $\Theta(T_0)\in\calK$, has already been proved.

Now suppose that $\Theta(T_u)\in\calK$ for $0<u<|K|$ and let $k\doteq 2r+s+1+u$ be the last entry of $K_u$. By Lemma \ref{lemmaRestrictionInvariance} we have $\Theta(T_u)\in\calK_k\subset\calK_{k+1}$, so we apply $e\doteq e_{k+1}+e_{k+2}$ to $\Theta(T_u)$. By the definition of $e_{k+1}$ and $e_{k+2}$ (for $n=k+1$) we see that the operator $e$ produce two tableaux from each tableau in $\Theta(T_u)$ by adding the entry $k+1$ in the first and the second row.
\noindent It is clear that
$$
\alpha^{T_{u+1}}(J',H',j_{k+1}(K_1'),K_2')=\alpha^{T_{u+1}}(J',H',K_1',j_{k+1}(K_2'))=\alpha^{T_u}(J',H',K_1',K_2')
$$
by the definition. Moreover $\calI(T_{u+1})$ is the set of tableaux obtained from $\calI(T_u)$ by applying $e$. So we conclude $\Theta(T_{u+1})=e\cdot\Theta(T_u)\in\calK_{k+1}\subset\calK$ and this finishes the proof of the theorem.
\end{proof}

So we may state our result in terms of shuffling relations.

\begin{corollary}\label{corollaryMain}
For any non standard tableau $T$ with two standard rows, $\Theta(T)$ is a shuffling relation for $T$; the set of all such shuffling relations generate the kernel of the multiplication map $\sfS^2\sfA_1\longrightarrow\sfA_2$ and the ideal of relations defining $\LGr$ in $\mP(V_{\omega_{n+1}}^*)$.
\end{corollary}

One may hope to simplify $\Theta(T)$ by considering only the tableaux with $K_1'=\varnothing$, in analogy with the shuffling relations for determinants. This is not possible as the following example for $n=4$ shows.

If $T=\binom{23}{14}$ we have:
$$
\Theta(T)=x\binom{23}{14}-x\binom{13}{24}+x\binom{12}{34}-x\binom{123}{4\quad}+x\binom{234}{1\quad}-x\binom{134}{2\quad}+x\binom{124}{3\quad}-x\binom{1234}{\varnothing},
$$
$$
\Theta\binom{234}{1\quad}=x\binom{234}{1\quad}-x\binom{134}{2\quad}+x\binom{124}{3\quad}-x\binom{123}{4\quad}.
$$
So, if the sum $\overline{\Theta}(T)$ of all elements with $K_1'=\varnothing$ in $\Theta(T)$ (i.e. of those with $4$ in the bottom row) were an element of $\calK$, then we have also
$$
\Theta(T)-\overline{\Theta}(T)-\Theta\binom{234}{1\quad}=-x\binom{1234}{\varnothing}\in\calK
$$
that is impossible since $\binom{1234}{\varnothing}$ is standard.

\section{Conclusion and relations in arbitrary characteristic}

We begin by restating the main result in terms of pfaffians and we slightly generalize it to a commutative unitary base ring $R$. Let $\sfB_R =R[x_{ij}| 1\leq i < j \leq n+1 ]$. It is easy to check  that the pfaffians of an antisymmetric matrix as defined in Section \ref{sectionPfaffians} is a polynomial in the variables $x_{ij}$ with integer coefficients so we can consider the elements $\pf_I$ as elements of $\sfB_R$. For a tableau $T$, let $\pf_T \in \sfB$ be the product of the pfaffians which appear in $T$. Since, as we have already noticed $\pf_{ij}=x_{ij}$, the ring $\sfB_R$ is spanned as an $R$-module by the elements $\pf_T$. Notice that $\Theta(T)$ has integral coefficients hence we may define $\Theta_\pf(T)$ as the element of $\sfB_R$ obtained by mapping $x(I)\longmapsto\pf_I$ for $I\in\rows{n}$ and $I\neq \emptyRow$, and mapping $x(\emptyRow)$ to $1$.

\begin{theorem}
Let $R$ be a commutative unitary ring.
\begin{enumerate}
\item The set of $\pf_T$ with $T$ standard and $T$ not containing the empty row is a $R$--basis of $\sfB_R$;
\item the ring $\sfB_R$ is defined by quadratic (but not necessarily homogenous) relations in the generators $\pf_I$;
\item for all non standard $T$ with two rows we have $\Theta_\pf(T) =0$ and these equations generate (as a $R$--module) the set of quadratic relations.
\end{enumerate}
\end{theorem}
\begin{proof}
For $R=\C$ the result follows by Corollary \ref{corollaryMain} and Section \ref{section32}. For $R=\Z$ notice that $B_\Z$ is a subring of $B_\C$; in particular the polynomial $\Theta_\pf(T)$ vanish. Since $\pf_T$ has coefficient $1$ in $\Theta_\pf(T)$, this allows to write any element $\pf_T$ as a $\Z$--linear combination of pfaffians of standard tableaux, clearly we may assume also that the empty row does not appear. Since the polynomials $\pf_T$ with $T$ as in (1) are linear indipendent over $\C$ we have proved that they are a $\Z$--basis, proving (1) for $\Z$.

Now (2) and (3) for $\Z$ follows by a standard argument. Indeed consider the ring $\sfC$ generated over $\Z$ by indeterminates $t_I$ with $I\in\standardRows{n},\,I\neq\varnothing$ module the ideal generated by the polynomials $\Theta^C(T)$ where $T$ runs in the set of non standard tableau with two standard rows obtained by mapping $x(I)$ to $t_I$ and $x(\varnothing)$ to $1$ in $\Theta(T)$.

Arguing as above we see that also $C$ has a basis given by the set of $t_T$ with $T$ running in the set of standard tableaux without the empty row. So the map $t_I\longmapsto \pf_I$ is an isomorphism and (2) and (3) follows.

The general case follows by $B_R=B_\Z\otimes_\Z R$.
\end{proof}

In a similar way we generalize Corollary \ref{corollaryMain} from $\C$ to $R$ with $R$ a field or the integers. The group $G$, the variety $\LGr$, and the line bundle $\calL$ may be defined in a flat way over the integers, hence they may be defined over $R$ and we denote by $G_R$, $\LGr_R$ and $\calL_R$ the associated schemes and line bundles. Let us define $A_R$ as $\bigoplus_{m\geq 0}H^0(\LGr_R,\calL_R^m)$. As proved in Remark $7$ in \cite{SeshadriP}, $A_\Z$ is generated in degree $1$ and if $R$ is a field then $A_R$ is isomorphic to $A_\Z\otimes_\Z R$. So arguing as in the proof of the previous theorem we have the following result.

\begin{theorem}Let $R$ be the set of integers or a field.
\begin{enumerate}
\item The sections $y(T)$ with $T$ standard are an $R$--basis for the ring  $\sfA_R$;
\item the ring $\sfA_R$ is defined by quadratic (but not necessarily homogenous) relations in the generators $x(I)$;
\item for all non standard $T$ with two rows we have $\Theta(T) =0$ and these equations generate (as a $R$--module) the set of quadratic relations.
\end{enumerate}
\end{theorem}

\bibliographystyle{plain}

\end{document}